\documentclass[11pt]{amsart}
\usepackage[utf8]{inputenc}
\usepackage{setspace, indentfirst}
\usepackage[margin=1in]{geometry}
\usepackage{cite, hyperref, comment, url}
\usepackage{tikz, graphicx, marvosym, subfigure, float}
\usetikzlibrary{decorations.pathreplacing}
\usepackage{dsfont, amsfonts, amssymb, amsmath, amsthm, enumitem, comment}
\usepackage{algorithm}
\usepackage[noend]{algpseudocode}
\numberwithin{equation}{section}
\newtheorem{theorem}{Theorem}[section]

\newtheorem{definition}[theorem]{Definition}
\newtheorem{lemma}[theorem]{Lemma}
\newtheorem{proposition}[theorem]{Proposition}
\newtheorem{corollary}[theorem]{Corollary}

\newtheorem{question}{Question}
\newcommand{\overbar}[1]{\mkern 1.7mu\overline{\mkern-1.7mu#1\mkern-1.7mu}\mkern 1.7mu}

\title{Chromatic Ramsey numbers and two-color Tur\'{a}n densities}

\author{Maria Axenovich}

\author{Simon Gaa}

\author{Dingyuan Liu}

\address{Karlsruhe Institute of Technology, Englerstraße 2, D-76131 Karlsruhe, Germany}

\email{maria.aksenovich@kit.edu, simon.gaa@kit.edu, liu@mathe.berlin}

\begin{document}
\begin{abstract}
Given a graph $G$, its $2$-color Tur\'{a}n number $\mathrm{ex}^{(2)}(n,G)$ is the maximum number of edges in an $n$-vertex graph, such that the edges can be colored with two colors avoiding a monochromatic copy of $G$. Let $\pi^{(2)}(G)=\lim_{n\to\infty}\mathrm{ex}^{(2)}(n,G)/\binom{n}{2}$ be the $2$-color Tur\'{a}n density of $G$. What real numbers in the interval $(0,1)$ are realized as the $2$-color Tur\'{a}n density of some graph? It is known that $\pi^{(2)}(G)=1-(R_{\chi}(G)-1)^{-1}$, where $R_{\chi}(G)$ is the chromatic Ramsey number of $G$. Burr, Erd\H{o}s, and Lov\'{a}sz showed that $(k-1)^2+1\leq{R_{\chi}(G)}\leq{R(k)}$, for any $k$-chromatic graph $G$, where $R(k)$ is the classical Ramsey number. However, it is an open problem to determine how many distinct values between $(k-1)^{2}+1$ and $R(k)$ can be realized as $R_{\chi}(G)$ of some $k$-chromatic graph $G$ for general $k$. In this paper, among others, we prove that there are $\Omega(k)$ different values of $R_{\chi}(G)$ among $k$-chromatic graphs $G$. This sheds more light onto the possible $2$-color Tur\'{a}n densities of graphs.
\end{abstract}
\maketitle

\section{Introduction}
Let $G$ be a graph containing at least one edge. The \textit{chromatic Ramsey number} $R_{\chi}(G)$ is defined as the smallest integer $N$, such that for some graph $F$ with $\chi(F)=N$, any $2$-coloring of the edges of $F$ results in a monochromatic copy of $G$. This notion was introduced by Burr, Erd\H{o}s, and Lov\'{a}sz~\cite{burr1976ramsey}, who further showed that if $\chi(G)=k$, then 
\begin{equation}
\label{equ1}
(k-1)^2+1\leq{R_{\chi}(G)}\leq{R(k)},
\end{equation}
where $R(k)$ is the classical Ramsey number. It is not difficult to see that the upper bound above can be attained by a clique on $k$ vertices. The lower bound is also attained by a graph given by Zhu~\cite{zhu2011fractional}.

The chromatic Ramsey number is closely related to the \textit{$2$-color Tur\'{a}n density} of $G$, defined as \[\pi^{(2)}(G)=\lim_{n\to\infty}\mathrm{ex}^{(2)}(n,G)/\binom{n}{2},\] where $\mathrm{ex}^{(2)}(n,G)$ is the maximum number of edges in an $n$-vertex graph, such that the edges can be colored with two colors avoiding a monochromatic copy of $G$. As the ordinary Tur\'{a}n density of $G$ is expressible in terms of its chromatic number $\chi(G)$, the function $\pi^{(2)}(G)$ is determined by the chromatic Ramsey number of $G$. The correspondence between $\pi^{(2)}(G)$ and $R_{\chi}(G)$ was established for cliques by Bialostocki, Caro, and Roditty~\cite{bialo1990zero} and for general graphs by Hancock, Staden, and Treglown~\cite{hancock2019ramsey}, who observed that
\begin{equation*}
\pi^{(2)}(G)=1-(R_{\chi}(G)-1)^{-1}.
\end{equation*}
This implies that determining $2$-color Tur\'{a}n densities boils down to understanding chromatic Ramsey numbers. Nevertheless, current knowledge of chromatic Ramsey numbers remains rather limited. It is even unclear how many distinct values between $(k-1)^{2}+1$ and $R(k)$ can be realized as $R_{\chi}(G)$ of some $k$-chromatic graph $G$ for general $k$. Here, we make the first attempt in this direction by providing the following lower bound on the number of such values.
\begin{theorem}
\label{cor1}
$\big\lvert\left\{R_{\chi}(G):\,\chi(G)=k\right\}\big\rvert=\Omega(k)$.
\end{theorem}

We also study $\left\{R_{\chi}(G):\,\chi(G)=k\right\}$ for small values of $k$. It is immediate from~\eqref{equ1} that $R_{\chi}(G)=2$ when $\chi(G)=2$ and $R_{\chi}(G)\in\{5,6\}$ when $\chi(G)=3$. When $\chi(G)=4$, we have that $10\leq{R_{\chi}(G)}\leq18$. However, it is unknown whether every integer between $10$ and $18$ is realizable as $R_{\chi}(G)$ for some $4$-chromatic graph $G$. Paul and Tardif~\cite{paul2012ramsey} showed that the odd wheels are $4$-chromatic graphs with chromatic Ramsey number $14$. Later, more $4$-chromatic graphs $G$ with $R_{\chi}(G)=14$ were found by Tardif~\cite{tardif2023ramsey}. Employing a computer-assisted proof, we find a new value realizable as the chromatic Ramsey numbers of $4$-chromatic graphs.
\begin{theorem}
\label{thm3}
There is a graph $G$ with $\chi(G)=4$ and $R_{\chi}(G)=11$.
\end{theorem}
Theorem~\ref{thm3} and the aforementioned results establish that
\begin{equation*}
\{10,11,14,18\}\subseteq\{R_{\chi}(G):\,\chi(G)=4\}\subseteq\{10,11,\dots,18\}.
\end{equation*}

Recall that the upper bound in~\eqref{equ1} can be attained by $G=K_{k}$. One might expect that every $k$-chromatic graph $G$ with $R_{\chi}(G)=R(k)$ contains a dense $k$-chromatic subgraph. However, our next result shows that there are $k$-chromatic graphs with arbitrarily large girth and chromatic Ramsey number $R(k)$. This establishes an analogue of the classical result by Erd\H{o}s~\cite{erdos1959girth} on graphs with high chromatic number and high girth. Since an $n$-vertex graph of girth $g$ contains at most $n^{1+2/(g-1)}$ edges (see, e.g., F\"uredi and Simonovits~\cite[Theorems 4.1 and 4.4]{furedi2013girth}), Theorem~\ref{thm2} shows that there are $k$-chromatic graphs with chromatic Ramsey number $R(k)$, which are sparse everywhere.
\begin{theorem}
\label{thm2}
For any $k,g\in\mathbb{N}$, there exists a graph $G$ with $\chi(G)=k$ and girth at least $g$, such that $R_{\chi}(G)=R(k)$.
\end{theorem}

Our main Theorem~\ref{cor1} is an immediate consequence of the following more precise result on the distribution of values of $R_{\chi}(G)$ in the interval $\{(k-1)^2+1,\dots,R(k)\}$. Throughout this paper, all logarithms are taken base $2$.
\begin{theorem}
\label{thm4}
There is an absolute constant $C>0$ such that the following holds. For every integer $k\geq2$ and any $\gamma$ satisfying $(k-1)^{2}+1\leq\gamma\leq{R(k)}$, there is a graph $G$ with $\chi(G)=k$, such that
\begin{equation*}
C\gamma\cdot\max\left\{\frac{\log{\gamma}}{k\log{k}},\,\left(\frac{\log{\gamma}}{k}\right)^{2}\right\}<R_{\chi}(G)\leq\gamma.
\end{equation*}
\end{theorem}
Note that when $\gamma$ is exponential in $k$, the upper and lower bounds in Theorem~\ref{thm4} are off by a multiplicative constant. Hence, Theorem~\ref{cor1} simply follows from the lower bound $R(k)\geq 2^{k/2}$ by Erd\H{o}s~\cite{erdos1947ramsey}.

The organization of this paper is as follows. In Section~\ref{preliminaries} we give necessary definitions and state basic properties of graph homomorphisms, chromatic Ramsey numbers, and fractional chromatic numbers. In Section~\ref{growth} we study the change of the Ramsey chromatic number when deleting a vertex, which could be of independent interest. Section~\ref{proof1} is devoted to proving Theorem~\ref{thm4} and Theorem~\ref{cor1}. In Sections~\ref{proof3} and~\ref{proof2} we establish Theorem~\ref{thm3} and Theorem~\ref{thm2}, respectively. Section~\ref{conclusions} includes some remarks and open problems.

\section{Definitions and preliminary results}
\label{preliminaries}
\subsection{General definitions, homomorphisms}
All graphs considered in this paper are simple and finite. Given a graph $G$, we denote its vertex set and edge set by $V(G)$ and $E(G)$, respectively. Let $\lvert{G}\rvert$ denote the number of vertices of $G$ and $\lVert{G}\rVert$ denote the number of edges. The \textit{girth} of a graph $G$, denoted by $\textsl{g}(G)$, is the shortest length of a cycle in $G$. Note that $\textsl{g}(G)=\infty$ if $G$ contains no cycle. The \textit{chromatic number} $\chi(G)$ of a graph $G$ is the smallest integer $k$, such that one can color the vertices of $G$ with $k$ colors so that every two adjacent vertices obtain different colors. For convenience, we say that $G$ is \textit{$k$-chromatic} if $\chi(G)=k$. We also abbreviate a coloring with $k$ colors as a \textit{$k$-coloring}.

The \textit{complement} $\overbar{G}$ of $G$ is the graph on the vertex set $V(G)$, whose edge set is $\binom{V(G)}{2}\backslash{E(G)}$. We say that a graph is a \textit{copy} of $G$ if it can be obtained by relabeling the vertices of $G$. A \textit{clique} is a graph, in which every two distinct vertices are adjacent. The \textit{clique number} $\omega(G)$ is the number of vertices of a largest clique contained in $G$. For any $n\in\mathbb{N}$, we let $K_{n}$ denote the clique on the vertex set $[n]:=\{1,\dots,n\}$. The \textit{classical Ramsey number} $R(k)$ is the smallest integer $N$ such that any $2$-coloring of the edges of $K_{N}$ results in a monochromatic copy of $K_{k}$.

The \textit{tensor product} $G\otimes{H}$ of two graphs $G$ and $H$ is defined as a graph on the vertex set $V(G)\times{V(H)}$, where two vertices $(u,v),(u',v')\in{V(G\otimes{H})}$ are adjacent if and only if $uu'\in{E(G)}$ and $vv'\in{E(H)}$.

The \textit{join} $G\vee{H}$ of two graphs $G$ and $H$ is a graph obtained by taking vertex-disjoint copies of $G$ and $H$ and adding an edge between every vertex of the copy of $G$ and every vertex of the copy of $H$. It is easy to see that $\chi(G\vee{H})=\chi(G)+\chi(H)$.

A \textit{homomorphism} from a graph $G$ to a graph $H$ is a mapping $f:V(G)\to{V(H)}$, such that $f(u)f(v)\in{E(H)}$ whenever $uv\in{E(G)}$. We say in this case that $H$ is a \textit{homomorphic image} of $G$. The \textit{homomorphic class} $\mathrm{Hom}(G)$ is defined as the family of all homomorphic images of $G$. For other graph theoretic notions we refer to the book by Diestel~\cite{diestel2017graph}.

We list some useful properties of graph homomorphisms and tensor products. All the properties here are basic and can be easily derived from their definitions. Let $G$ and $H$ be graphs, then

\vspace{1em}\begin{tabular}{ll}\vspace{1em}
\textbf{(P1)} $K_{\chi(G)}\in\mathrm{Hom}(G)$; & \textbf{(P2)} if $H\in\mathrm{Hom}(G')$ and $G'\in\mathrm{Hom}(G)$, then $H\in\mathrm{Hom}(G)$;\\
\textbf{(P3)} $\{G,H\}\subseteq\mathrm{Hom}(G\otimes{H})$; & \textbf{(P4)} $\chi(G\otimes{H})\leq\min\{\chi(G),\chi(H)\}$.
\end{tabular}\vspace{1em}

\subsection{$\boldsymbol{R_{\chi}(G)}$ and $\boldsymbol{R\left(\mathrm{Hom}(G)\right)}$}
\label{ramsey}
Let $\mathcal{F}$ be a family of graphs and define $R(\mathcal{F})$ as the smallest integer $N$, such that any $2$-coloring of $E(K_{N})$ results in a monochromatic copy of a graph from $\mathcal{F}$. Burr, Erd\H{o}s, and Lov\'{a}sz~\cite{burr1976ramsey} proved that for any graph $G$, the values of $R_{\chi}(G)$ and $R\left(\mathrm{Hom}(G)\right)$ actually coincide. Recall that $\mathrm{Hom}(G)$ is the homomorphic class of $G$.
\begin{proposition}[Burr--Erd\H{o}s--Lov\'{a}sz~\cite{burr1976ramsey}]
\label{pro2}
$R_{\chi}(G)=R\left(\mathrm{Hom}(G)\right)$ for all graphs $G$.
\end{proposition}

\subsection{Fractional chromatic numbers}
There are several equivalent definitions of the fractional chromatic number, here we follow the one by Godsil and Royle~\cite{godsil2001algebraic}. We say that a set $I\subseteq{V(G)}$ is \textit{independent} if $I$ contains no two adjacent vertices. The \textit{independence number} $\alpha(G)$ is defined as the largest size of an independent set in $G$. Let $\mathcal{I}(G)$ denote the family of all independent sets in $G$. For $v\in{V(G)}$, let $\mathcal{I}_{v}(G)$ denote the family of all independent sets in $G$ containing $v$. A \textit{fractional coloring} of $G$ is a mapping $c_{f}$ from $\mathcal{I}(G)$ to $\mathbb{R}_{\geq0}$, such that for every $v\in{V(G)}$ we have $\sum_{I\in\mathcal{I}_{v}(G)}c_{f}(I)\geq1$. The \textit{fractional chromatic number} $\chi_{f}(G)$ of a graph $G$ is defined as the minimum value of $\sum_{I\in\mathcal{I}(G)}c_{f}(I)$ among all fractional colorings $c_{f}$ of $G$.

The following property of fractional chromatic numbers has been extensively mentioned in the literature, see, e.g.,~\cite{godsil2001algebraic}.
\begin{lemma}
\label{lem1}
Let $G$ be a graph. Then
\begin{equation*}
\frac{\lvert{G}\rvert}{\alpha(G)}\leq\chi_{f}(G)\leq\chi(G).
\end{equation*}
\end{lemma}
In particular, Lemma~\ref{lem1} yields that $\chi_{f}(K_{n})=n$ for all $n\in\mathbb{N}$.

A long-standing conjecture of Hedetniemi~\cite{hedetniemi1966conjecture} claimed that $\chi(G\otimes{H})=\min\left\{\chi(G),\chi(H)\right\}$ for any graphs $G$ and $H$. This conjecture was recently disproved by Shitov~\cite{shitov2019counter}. Nevertheless, Zhu~\cite{zhu2011fractional} showed that the fractional version of Hedetniemi’s conjecture is true.
\begin{proposition}[Zhu~\cite{zhu2011fractional}]
\label{pro3}
Let $G$ and $H$ be graphs. Then
\begin{equation*}
\chi_{f}(G\otimes{H})=\min\left\{\chi_{f}(G),\chi_{f}(H)\right\}.
\end{equation*}
\end{proposition}

\section{Chromatic Ramsey numbers upon vertex deletion}
\label{growth}
Let $G$ be a graph containing at least one edge and $F$ be a graph obtained by deleting a vertex from $G$. An open problem in Ramsey theory is to understand how the Ramsey number of $F$ changes relative to the Ramsey number of $G$. In terms of classical Ramsey numbers, Conlon, Fox, and Sudakov~\cite{conlon2020short} showed that $R(F)\geq\Omega(d/\log(1/d))R(G)$, where $d=d(G):=\lVert{G}\rVert/\binom{\lvert{G}\rvert}{2}$ denotes the density of $G$. This in particular implies that for any sufficiently dense $G$, the value of $R(F)$ is at least a constant fraction of $R(G)$. Furthermore, Wigderson~\cite{wigderson2024ramsey} recently obtained that $R(F)\geq\Omega(1/\sqrt{\lvert{G}\rvert\log{\lvert{G}\rvert}})R(G)$, which is a better bound when $G$ is sparse. In this section we prove analogous results regarding chromatic Ramsey numbers.
\begin{proposition}
\label{pro1}
There is an absolute constant $C_{1}>0$ such that the following holds. Let $G$ be a graph containing at least one edge and $F$ be a graph obtained by deleting a vertex from $G$. For any $\delta\geq2$ with
\begin{equation*}
\frac{\delta}{\log{\delta}}\geq\frac{\chi(G)}{\log{R_{\chi}(G)}},
\end{equation*}
we have 
\begin{equation*}
R_{\chi}(F)\geq\frac{R_{\chi}(G)}{C_{1}\delta}.
\end{equation*}
\end{proposition}

To prove Proposition~\ref{pro1}, we make use of a result of Erd\H{o}s and Szemer\'{e}di~\cite{erdos1972ramsey}. Let $\epsilon\in(0,1/2]$ and $G$ be a graph. We say that a coloring of $E(G)$ is \textit{$\epsilon$-balanced} if the number of edges of each color is at least $\epsilon\lVert{G}\rVert$.
\begin{lemma}[Erd\H{o}s--Szemer\'{e}di~\cite{erdos1972ramsey}]
\label{lem2}
There is an absolute constant $A>0$ such that the following holds for any $\epsilon\in(0,1/2]$ and every integer $N\geq1/\epsilon$. If a $2$-coloring of $E(K_{N})$ is not $\epsilon$-balanced, then there exists a monochromatic copy of $K_{n}$ with $n\geq\frac{A\log{N}}{\epsilon\log(1/\epsilon)}$.
\end{lemma}
Note that Lemma~\ref{lem2} was stated in the original paper without quantifiers on $N$ and $\epsilon$, but as the anonymous referee pointed out (see also~\cite{liu2026}), it is sufficient to assume that $N \geq 1/\epsilon$.

\begin{proof}[Proof of Proposition~\ref{pro1}]
Let $A>0$ be the absolute constant given in Lemma~\ref{lem2}. Let $C_{1}\geq64$ be a sufficiently large constant such that $A\sqrt{C_{1}}/8\geq1$. We do not attempt to optimize the constants in this proof. Suppose for the purpose of contradiction that $R_{\chi}(G)>C_{1}\delta{R_{\chi}(F)}$.

Let $N=R_{\chi}(G)-1$ and color $E(K_{N})$ with red and blue, such that there is no monochromatic copy of any graph from $\mathrm{Hom}(G)$. For $v,w\in{V(K_{N})}$, we say that $w$ is a \textit{red neighbor} of $v$ if the edge $vw$ is colored red, otherwise $w$ is called a \textit{blue neighbor} of $v$. Take an arbitrary vertex $v\in{V(K_{N})}$, without loss of generality assume that $v$ has at least $(N-1)/2$ red neighbors. Let $Q$ be the clique induced by all the red neighbors of $v$. Since $R_{\chi}(G)>C_{1}\delta{R_{\chi}(F)}\geq128$, we have that
\begin{equation}
\label{equ3}
\lvert{Q}\rvert\geq(N-1)/2\geq{R_{\chi}(G)/3}.
\end{equation}
Let $\epsilon=4/(C_{1}\delta)<1/2$. Now we look at the coloring of $E(Q)$.\\
\\\textbf{Case 1:} the coloring of $E(Q)$ is $\epsilon$-balanced.\\
As the number of blue edges is at least $\epsilon\lVert{Q}\rVert$, there must exist a vertex $w\in{V(Q)}$ with at least
\begin{equation*}
\frac{2\epsilon\lVert{Q}\rVert}{\lvert{Q}\rvert}=\epsilon(\lvert{Q}\rvert-1)\geq\epsilon\cdot\frac{C_{1}\delta{R_{\chi}(F)}}{4}=R_{\chi}(F)
\end{equation*}
blue neighbors. Here, we used~\eqref{equ3} and the assumption that $R_{\chi}(G)>C_{1}\delta{R_{\chi}(F)}$ to derive the above inequality. Denote $Q'$ the clique induced by all the blue neighbors of $w$ in $Q$. Since $\lvert{Q'}\rvert\geq{R_{\chi}(F)}$, there exists a monochromatic copy of some $F'\in\mathrm{Hom}(F)$ in $Q'$. Together with $v$ or $w$ we obtain a monochromatic copy of $K_{1}\vee{F'}\in\mathrm{Hom}(G)$, a contradiction.\\
\\\textbf{Case 2:} the coloring of $E(Q)$ is not $\epsilon$-balanced.\\
Recall that $\epsilon=4/(C_{1}\delta)$ and $\lvert{Q}\rvert\geq{R_{\chi}(G)/3}>(C_1\delta)/3>1/\epsilon$. Then by Lemma~\ref{lem2} there exists a monochromatic copy of $K_{n}$ in $Q$ with $n\geq A\log \lvert{Q}\rvert/(\epsilon\log(1/\epsilon))$. Thus 
\begin{equation*}
n\geq\frac{A\log{\lvert{Q}\rvert}}{\epsilon\log(1/\epsilon)}\geq\frac{A\log{R_{\chi}(G)}}{2\epsilon\log(1/\epsilon)}
\geq\frac{A\log{R_{\chi}(G)}}{\frac{8\log(C_{1}\delta)}{C_{1}\delta}}\geq\frac{A\log{R_{\chi}(G)}}{\frac{8\log{\delta}}{\sqrt{C_{1}}\delta}}=\frac{(A\sqrt{C_{1}}/8)\log{R_{\chi}(G)}}{\frac{\log{\delta}}{\delta}}\geq\chi(G).
\end{equation*}
Here, in the second inequality we used that $\lvert{Q}\rvert\geq{R_{\chi}(G)/3}\geq\sqrt{R_{\chi}(G)}$, in the fourth inequality we used that $\log(C_{1}\delta)\leq\sqrt{C_{1}}\log\delta$ holds for all $C_{1}\geq64$ and $\delta\geq2$, and in the last inequality we used that $\delta/\log{\delta}\geq\chi(G)/\log{R_{\chi}(G)}$ and $A\sqrt{C_{1}}/8\geq1$.
Since $n\geq\chi(G)$, there is a monochromatic copy of $K_{\chi(G)}$ in $Q$, which by \textbf{(P1)} is a homomorphic image of $G$, a contradiction.
\end{proof}

From Proposition~\ref{pro1} we can derive the following corollary, which shall be used in our proof of Theorem~\ref{thm4}.
\begin{corollary}
\label{cor2}
There is an absolute constant $C_{2}>0$ such that the following holds. Let $F$ be a graph containing at least one vertex and $G=K_{1}\vee{F}$. Then 
\begin{equation*}
R_{\chi}(F)+1\leq{R_{\chi}(G)}\leq{C_{2}\cdot\min\left\{\frac{\chi(G)\log{\chi(G)}}{\log{R_{\chi}(G)}},\,\left(\frac{\chi(G)}{\log{R_{\chi}(G)}}\right)^{2}\right\}}\cdot{R_{\chi}(F)}.
\end{equation*}
\end{corollary}
\begin{proof}[Proof of Corollary~\ref{cor2}]
Let
\begin{equation*}
\delta=8\cdot\min\left\{\frac{\chi(G)\log{\chi(G)}}{\log{R_{\chi}(G)}},\,\left(\frac{\chi(G)}{\log{R_{\chi}(G)}}\right)^{2}\right\}.
\end{equation*}
From~\eqref{equ1} and the upper bound on $R(k)$ by Erd\H{o}s and Szekeres~\cite{erdos1935ramsey} we have $R_{\chi}(G)\leq R(\chi(G))\leq4^{\chi(G)}$, which implies both $\frac{\chi(G)\log{\chi(G)}}{\log{R_{\chi}(G)}}\geq\frac{\log{\chi(G)}}{2}\geq1/4$ and $\left(\frac{\chi(G)}{\log{R_{\chi}(G)}}\right)^{2}\geq1/4$. Namely, $\delta\geq2$. Now we show that $\frac{\delta}{\log{\delta}}\geq\frac{\chi(G)}{\log{R_{\chi}(G)}}$ holds. If $\delta=8\cdot\frac{\chi(G)\log{\chi(G)}}{\log{R_{\chi}(G)}}$, then we have $\delta\leq8\chi(G)$ since $\chi(G)\leq{R_{\chi}(G)}$. Therefore,
\begin{equation*}
\frac{\delta}{\log{\delta}}\geq\frac{8\chi(G)\log{\chi(G)}}{\log{R_{\chi}(G)}\log\left(8\chi(G)\right)}\geq\frac{\chi(G)}{\log{R_{\chi}(G)}}.
\end{equation*}
If $\delta=8\cdot\left(\frac{\chi(G)}{\log{R_{\chi}(G)}}\right)^{2}$, then
\begin{equation*}
\frac{\delta}{\log{\delta}}\geq\sqrt{\delta}\geq\frac{\chi(G)}{\log{R_{\chi}(G)}},
\end{equation*}
where we used that $\log{\delta}\leq\sqrt{\delta}$ holds for all $\delta\geq1$. Then, by Proposition~\ref{pro1} there exists an absolute constant $C_{1}>0$ such that $R_{\chi}(G)\leq{C_{1}\delta{R_{\chi}(F)}}$. The second inequality in Corollary~\ref{cor2}, that is,
\begin{equation*}
R_{\chi}(G)\leq{C_{2}\cdot\min\left\{\frac{\chi(G)\log{\chi(G)}}{\log{R_{\chi}(G)}},\,\left(\frac{\chi(G)}{\log{R_{\chi}(G)}}\right)^{2}\right\}}\cdot{R_{\chi}(F)},
\end{equation*}
follows by letting $C_{2}=8C_{1}$.

Now we prove that $R_{\chi}(F)+1\leq{R_{\chi}(G)}$. From Proposition~\ref{pro2} we have $R_{\chi}(G)=R\left(\mathrm{Hom}(G)\right)$. It suffices to find a $2$-coloring of $E(K_{N})$ that contains no monochromatic copy of any graph from $\mathrm{Hom}(G)$, where $N=R_{\chi}(F)$. Fix an arbitrary vertex $v\in{V(K_{N})}$. Due to the value of $N$, there exists a $2$-coloring of $E(K_{N}-v)$ containing no monochromatic copy of any graph from $\mathrm{Hom}(F)$. We extend this $2$-coloring to $E(K_{N})$ by coloring the edges incident to the vertex $v$ arbitrarily. Let $G'\subseteq{K_{N}}$ be a copy of some graph from $\mathrm{Hom}(G)$. It suffices to show that $G'$ is not monochromatic. Observe that $G'$ must be the join of a single vertex $w\in{V(K_{N})}$ and some graph $F'\in\mathrm{Hom}(F)$. If $v\notin{V(G')}$ or $w=v$, then $F'\subseteq{K_{N}-v}$. According to our coloring such $F'$ is not monochromatic, thus, $G'$ is not monochromatic. If $v\in{V(G')}$ and $w\neq{v}$, then $v\in{V(F')}$ and $G'-v=\{w\}\vee(F'-v)$. Note that $G'-v$ is a subgraph of $K_{N}-v$. Since $w$ is adjacent to all vertices in $F'-v$, it follows that $G'-v$ contains a copy of $F'$ as a subgraph. However, there is no monochromatic copy of $F'$ in $K_{N}-v$, namely, $G'$ is not monochromatic.
\end{proof}

\section{Chromatic Ramsey numbers of $k$-chromatic graphs}
\label{proof1}
Recall that we have the bounds $(k-1)^{2}+1\leq{R_{\chi}(G)}\leq{R(k)}$ for every $k$-chromatic graph $G$. In order to prove Theorem~\ref{thm4}, for any given $\gamma\in\mathbb{N}$ with $(k-1)^{2}+1\leq\gamma\leq{R(k)}$, we want to construct a $k$-chromatic graph $G$, such that $R_{\chi}(G)$ is close to $\gamma$. Our idea of construction is rather natural, that is, we take the join of two graphs which achieve the upper bound and the lower bound, respectively. To this end, we shall first generalize a construction of Zhu~\cite{zhu2011fractional}.
\subsection{Generalized Zhu's graph $\boldsymbol{Z_{\ell,n}}$}
Solving a famous conjecture of Burr, Erd\H{o}s, and Lov\'{a}sz~\cite{burr1976ramsey}, for any positive integer $\ell$, Zhu~\cite{zhu2011fractional} constructed a graph with chromatic number $\ell$, whose chromatic Ramsey number achieves the lower bound $(\ell-1)^{2}+1$. Here, we generalize Zhu's construction and give the following definition.
\begin{definition}
Let $\ell,n\in\mathbb{N}$ with $\ell\leq{n}$. Let $\mathcal{F}(\ell,n)$ denote the family of all graphs $G$ satisfying $\chi_{f}(G)>\ell-1$ and $V(G)\subseteq[n]$. The generalized Zhu's graph $Z_{\ell,n}$ is defined as
\begin{equation*}
Z_{\ell,n}:=\bigotimes_{G\in\mathcal{F}(\ell,n)}G.
\end{equation*}
\end{definition}

The original construction of Zhu employs only those graphs $G$ in the product, whose number of vertices is exactly $(\ell-1)^2+1$. Recalling \textbf{(P3)}, it is not hard to see that the original Zhu's graph is a homomorphic image of any $Z_{\ell,n}$ with $n\geq(\ell-1)^{2}+1$. By slightly modifying the proof from~\cite{zhu2011fractional}, we shall see that $Z_{\ell,n}$ with $n\geq(\ell-1)^2+1$ is still an $\ell$-chromatic graph whose chromatic Ramsey number attains the lower bound. Note that \textbf{(P3)} yields that $\mathcal{F}(\ell,n)\subseteq\mathrm{Hom}\left(Z_{\ell,n}\right)$ and Proposition~\ref{pro3} implies that $\chi_{f}\left(Z_{\ell,n}\right)>\ell-1$.

\begin{lemma}
\label{lem3}
Let $\ell,n\in\mathbb{N}$ with $n\geq(\ell-1)^{2}+1$. Then we have $\chi\left(Z_{\ell,n}\right)=\ell$ and
\begin{equation*}
R_{\chi}\left(Z_{\ell,n}\right)=(\ell-1)^{2}+1.
\end{equation*}
\end{lemma}
\begin{proof}[Proof of Lemma~\ref{lem3}]
We first prove $\chi\left(Z_{\ell,n}\right)=\ell$. Note that we have $K_{\ell}\in\mathcal{F}(\ell,n)$ because $\chi_{f}(K_{\ell})=\ell>\ell-1$ and $V(K_{\ell})=[\ell]\subseteq[n]$. Then \textbf{(P4)} yields that $\chi\left(Z_{\ell,n}\right)\leq\chi(K_{\ell})=\ell$. On the other hand, by Lemma~\ref{lem1} and Proposition~\ref{pro3} we have
\begin{equation*}
\chi\left(Z_{\ell,n}\right)\geq\chi_{f}\left(Z_{\ell,n}\right)=\min\left\{\chi_{f}(G):\,G\in\mathcal{F}(\ell,n)\right\}>\ell-1,
\end{equation*}
implying that $\chi\left(Z_{\ell,n}\right)=\ell$. Now we prove that $R_{\chi}\left(Z_{\ell,n}\right)=(\ell-1)^{2}+1$. Since $\chi\left(Z_{\ell,n}\right)=\ell$, from~\eqref{equ1} we know already $R_{\chi}\left(Z_{\ell,n}\right)\geq(\ell-1)^{2}+1$. It remains to show $R_{\chi}\left(Z_{\ell,n}\right)\leq(\ell-1)^{2}+1$. Let $N=(\ell-1)^{2}+1$ and consider an arbitrary coloring of $E(K_{N})$ with red and blue. We shall argue that there is a monochromatic copy of a graph from $\mathrm{Hom}\left(Z_{\ell,n}\right)$. If there exists a red copy of $K_{\ell}$, then we are done. If not, then the graph $Q$ obtained from $K_{N}$ by removing all red edges has independence number at most $\ell-1$. Note that $Q$ is a blue graph and by Lemma~\ref{lem1} it follows that
\begin{equation*}
\chi_{f}(Q)\geq\frac{\lvert{Q}\rvert}{\alpha(Q)}\geq\frac{N}{\ell-1}>\ell-1.
\end{equation*}
Since $V(Q)=[(\ell-1)^{2}+1]\subseteq[n]$, we have $Q\in\mathcal{F}(\ell,n)\subseteq\mathrm{Hom}\left(Z_{\ell,n}\right)$.
\end{proof}

The main advantage of generalized Zhu's graphs is that we can easily control their homomorphic classes.
\begin{lemma}
\label{lem4}
Let $\ell,n\in\mathbb{N}$ with $\ell<n$. Then we have
\begin{equation*}
\mathrm{Hom}\left(Z_{\ell+1,n}\right)\subseteq\mathrm{Hom}\left(Z_{\ell,n}\right).
\end{equation*}
\end{lemma}
\begin{proof}[Proof of Lemma~\ref{lem4}]
Since $\ell<n$, the families $\mathcal{F}(\ell,n)$ and $\mathcal{F}(\ell+1,n)$ are non-empty. Moreover, from the definition it follows that $\mathcal{F}(\ell+1,n)$ is a subfamily of $\mathcal{F}(\ell,n)$. Hence, we can write $Z_{\ell,n}$ as
\begin{equation*}
\bigotimes_{G\in\mathcal{F}(\ell,n)}G=\left(\bigotimes_{G\in\mathcal{F}(\ell+1,n)}G\right)\otimes\left(\bigotimes_{G\in\mathcal{F}(\ell,n)\backslash\mathcal{F}(\ell+1,n)}G\right)=Z_{\ell+1,n}\otimes\left(\bigotimes_{G\in\mathcal{F}(\ell,n)\backslash\mathcal{F}(\ell+1,n)}G\right),
\end{equation*}
which by \textbf{(P3)} implies that $Z_{\ell+1,n}\in\mathrm{Hom}\left(Z_{\ell,n}\right)$. Then, by \textbf{(P2)} we know that every homomorphic image of $Z_{\ell+1,n}$ is also a homomorphic image of $Z_{\ell,n}$, namely, $\mathrm{Hom}\left(Z_{\ell+1,n}\right)\subseteq\mathrm{Hom}\left(Z_{\ell,n}\right)$.
\end{proof}

\subsection{Proof of Theorem~\ref{thm4}}
\begin{proof}[Proof of Theorem~\ref{thm4}]
Let $C_{2}>0$ be the constant given by Corollary~\ref{cor2}. We may assume without loss of generality that $C_{2}>4$. Let $C=1/C_{2}<1/4$. Fix an integer $k\geq2$ and any $\gamma$ satisfying $(k-1)^{2}+1\leq\gamma\leq{R(k)}$. We need to show that there exists a graph $G$ with $\chi(G)=k$, such that
\begin{equation*}
C\gamma\cdot\max\left\{\frac{\log{\gamma}}{k\log{k}},\,\left(\frac{\log{\gamma}}{k}\right)^{2}\right\}<R_{\chi}(G)\leq\gamma.
\end{equation*}
Note that $C\gamma\cdot\max\left\{\frac{\log{\gamma}}{k\log{k}},\,\left(\frac{\log{\gamma}}{k}\right)^{2}\right\}$ is indeed strictly smaller than $\gamma$, because $\gamma\leq{R(k)}\leq4^{k}$ and thus $\max\left\{\frac{\log{\gamma}}{k\log{k}},\,\left(\frac{\log{\gamma}}{k}\right)^{2}\right\}\leq4$.

Let $n=k^{2}$. For $i\in[k]$ we define the graph $G_{i}:=K_{k-i}\vee{Z_{i,n}}$, in particular we have $G_{k}=Z_{k,n}$. It follows from Lemma~\ref{lem3} that $\chi(G_{i})=\chi(K_{k-i})+\chi(Z_{i,n})=k-i+i=k$ for all $i\in[k]$. Furthermore, we take $\ell\in[k]$ to be the smallest value of $i\in[k]$ satisfying $R_{\chi}(G_{i})\leq\gamma$. Such $\ell$ exists because we have $R_{\chi}(G_{k})=R_{\chi}\left(Z_{k,n}\right)=(k-1)^{2}+1\leq\gamma$. Now we show that $R_{\chi}(G_{\ell})$ is between $C\gamma\cdot\max\left\{\frac{\log{\gamma}}{k\log{k}},\,\left(\frac{\log{\gamma}}{k}\right)^{2}\right\}$ and $\gamma$. Due to the way we choose $\ell$, we already have that $R_{\chi}(G_{\ell})\leq\gamma$. It remains to prove $R_{\chi}(G_{\ell})>C\gamma\cdot\max\left\{\frac{\log{\gamma}}{k\log{k}},\,\left(\frac{\log{\gamma}}{k}\right)^{2}\right\}$.\\
\\\textbf{Case 1:} $\ell=1$.\\
Then $G_{\ell}=K_{k-1}\vee{Z_{1,n}}$ contains a copy of $K_{k}$, hence,
\begin{equation*}
R_{\chi}(G_{\ell})\geq{R_{\chi}(K_{k})}=R(k)\geq\gamma>C\gamma\cdot\max\left\{\frac{\log{\gamma}}{k\log{k}},\,\left(\frac{\log{\gamma}}{k}\right)^{2}\right\}.
\end{equation*} 
\\\textbf{Case 2:} $\ell\geq2$.\\
Then due to the minimality of $\ell$, we have
\begin{equation}
\label{equ4}
\gamma<R_{\chi}(G_{\ell-1}).
\end{equation}
Since $F:=K_{k-\ell}\vee{Z_{\ell-1,n}}$ has chromatic number $k-1\geq1$ and $G_{\ell-1}$ can be obtained by joining a vertex to $F$, applying Corollary~\ref{cor2} we have
\begin{equation}
\label{equ5}
\begin{aligned}
R_{\chi}(G_{\ell-1})&\leq{C_{2}\cdot\min\left\{\frac{\chi(G_{\ell-1})\log{\chi(G_{\ell-1})}}{\log{R_{\chi}(G_{\ell-1})}},\,\left(\frac{\chi(G_{\ell-1})}{\log{R_{\chi}(G_{\ell-1})}}\right)^{2}\right\}}\cdot{R_{\chi}(F)}\\
&\leq{C_{2}\cdot\min\left\{\frac{k\log{k}}{\log{\gamma}},\,\left(\frac{k}{\log{\gamma}}\right)^{2}\right\}}\cdot{R_{\chi}(F)}.
\end{aligned}
\end{equation}
By Lemma~\ref{lem4} it holds that $\mathrm{Hom}\left(Z_{\ell,n}\right)\subseteq\mathrm{Hom}\left(Z_{\ell-1,n}\right)$. Accordingly, we have $\mathrm{Hom}\left(K_{k-\ell}\vee{Z_{\ell,n}}\right)\subseteq\mathrm{Hom}\left(K_{k-\ell}\vee{Z_{\ell-1,n}}\right)$, which in turn implies that
\begin{equation}
\label{equ6}
R_{\chi}(F)=R_{\chi}\left(K_{k-\ell}\vee{Z_{\ell-1,n}}\right)\leq{R_{\chi}\left(K_{k-\ell}\vee{Z_{\ell,n}}\right)}=R_{\chi}(G_{\ell}).
\end{equation}
Combining~\eqref{equ4}~\eqref{equ5} and~\eqref{equ6}, we obtain
\begin{equation*}
\gamma<C_{2}\cdot\min\left\{\frac{k\log{k}}{\log{\gamma}},\,\left(\frac{k}{\log{\gamma}}\right)^{2}\right\}\cdot{R_{\chi}(G_{\ell})}.
\end{equation*}
Recall that $C=1/C_{2}$, so we have $R_{\chi}(G_{\ell})>C\gamma\cdot\max\left\{\frac{\log{\gamma}}{k\log{k}},\,\left(\frac{\log{\gamma}}{k}\right)^{2}\right\}$.
\end{proof}

\subsection{Proof of Theorem~\ref{cor1}}
\begin{proof}[Proof of Theorem~\ref{cor1}]
Assume that $k$ is sufficiently large. Let $0<C<1$ be the constant from Theorem~\ref{thm4} and $m=\lfloor\log_{64/C}2^{k/2}\rfloor=\Omega(k)$. We set $\gamma_{0}=2^{k/8}$ and $\gamma_i=(64/C)^i\gamma_0$ for each $i\in[m]$. Since $k$ is sufficiently large and $R(k)\geq2^{k/2}$~\cite{erdos1947ramsey}, it holds that \[(k-1)^{2}+1\leq\gamma_{0}<\gamma_1<\dots<\gamma_m\leq R(k).\] By Theorem~\ref{thm4} we know that for every $i\in[m]$, there is a $k$-chromatic graph $G$ with \[\gamma_{i}\geq{R_{\chi}(G)}>C\gamma_{i}\cdot\left(\frac{\log{\gamma_{i}}}{k}\right)^{2}>\frac{C\gamma_i}{64}\geq\gamma_{i-1}.\] Namely, $\left\{R_{\chi}(G):\,\chi(G)=k\right\}\cap(\gamma_{i-1},\gamma_{i}]\neq\emptyset$ for all $i\in[m]$, from which our lower bound follows.
\end{proof}

\section{Chromatic Ramsey numbers of $4$-chromatic graphs}
\label{proof3}
As aforementioned, the set $\left\{R_{\chi}(G):\,\chi(G)=k\right\}$ is fully characterized when $k\in[3]$. Now we turn our attention to the case of $k=4$ and give an idea on how to determine the chromatic Ramsey number of a given $4$-chromatic graph using computer assistance.
\subsection{The algorithm}
For $n,s,t\in\mathbb{N}$, a \textit{Ramsey$(s,t,n)$-graph} $F$ is an $n$-vertex graph with $\omega(F)<s$ and $\alpha(F)<t$. Note that the existence of a Ramsey$(s,t,n)$-graph implies that there is a coloring of $E(K_{n})$ with red and blue that contains neither a red copy of $K_{s}$ nor a blue copy of $K_{t}$. Moreover, a \textit{Ramsey$(s,t)$-graph} is a Ramsey$(s,t,n)$-graph for some $n\in\mathbb{N}$. The fundamental Ramsey’s theorem~\cite{ramsey1930logic} implies that the number of all Ramsey$(s,t)$-graphs (up to vertex relabeling) is finite. Our method here for finding chromatic Ramsey numbers of $4$-chromatic graphs makes use of the fact that not only the Ramsey number $R(4)$ is known~\cite{greenwood1955ramsey}, but also all Ramsey$(4,4)$-graphs have been determined and listed by McKay~\cite{mckay2023ramsey}.

First, we want to explain the idea behind the algorithm. Let $G$ be a graph with $\chi(G)=4$. From~\eqref{equ1} we have
\begin{equation}
\label{equ7}
10=(4-1)^{2}+1\leq{R_{\chi}(G)}\leq{R(4)}=18.
\end{equation}
Let $\mathrm{Hom}'(G)$ denote the family of minimal homomorphic images of $G$, namely,
\begin{equation*}
\mathrm{Hom}'(G):=\left\{H\in\mathrm{Hom}(G):\,H'\notin\mathrm{Hom}(G)\text{ for all proper subgraphs }H'\text{ of }H\right\}.
\end{equation*}
It is easy to see that $R\left(\mathrm{Hom}(G)\right)=R\left(\mathrm{Hom}'(G)\right)$, and thus Proposition~\ref{pro2} implies that $R_{\chi}(G)=R\left(\mathrm{Hom}'(G)\right)$. Namely, $R_{\chi}(G)$ is upper bounded by a positive integer $n$ if and only if for every $n$-vertex graph $F$ one can find a copy of a graph from $\mathrm{Hom}'(G)$ in $F$ or $\overbar{F}$. In fact, since $K_{4}\in\mathrm{Hom}'(G)$ by \textbf{(P1)}, it suffices to consider Ramsey$(4,4,n)$-graphs and check them and their complements for a copy of a graph from $\mathrm{Hom}'(G)\backslash\{K_{4}\}$. We make use of this observation and test all Ramsey$(4,4,n)$-graphs, starting from $n=17$ and then $n=16,\dots,10$. If the algorithm finds for the first time some Ramsey$(4,4,n)$-graph $F$, such that neither $F$ nor $\overbar{F}$ contains a copy of any graph from $\mathrm{Hom}'(G)\backslash\{K_{4}\}$, then we immediately obtain
\begin{equation*}
R_{\chi}(G)=n+1.
\end{equation*}
If, on the other hand, every check is successful, then by~\eqref{equ7} we conclude that $R_{\chi}(G)=10$. The advantage of this procedure stems from the fact that the number of Ramsey$(4,4,n)$-graphs is much smaller than the number of all $n$-vertex graphs. Although there are still over three million different Ramsey$(4,4,n)$-graphs for $n\in\{10,11,\dots,17\}$, we can use this algorithm to determine the exact chromatic Ramsey number of an input $4$-chromatic graph in reasonable time. Below we present the pseudocode of our algorithm.
\begin{algorithm}[H]
\caption{}\label{algorithm}
\begin{tabular}{ll}
\textbf{Input:} & a graph $G$ with $\chi(G)=4$\\
& lists $L(10),L(11),\dots,L(17)$ of Ramsey$(4,4)$-graphs on $10,11,\dots,17$ vertices\\
\textbf{Output:} & $R_{\chi}(G)\in\{10,11,\dots,18\}$
\end{tabular}
\begin{algorithmic}[1]
\For{$i\gets17$ \textbf{to} $10$}
\For{\textbf{each} $F\in{L(i)}$}
\State{$\textit{hit}\gets\text{false}$}
\For{\textbf{each} $H\in\mathrm{Hom'(G)}\backslash\{K_{4}\}$}\label{line4}
\If{\texttt{isSubgraph$(H,F)$} \textbf{or} \texttt{isSubgraph$(H,\overbar{F})$}}
\State{$\textit{hit}\gets\text{true}$}\label{line6}
\EndIf
\EndFor
\If{\textbf{not} \textit{hit}}
\State\Return{$i+1$}
\EndIf
\EndFor
\EndFor
\State\Return{$10$}
\end{algorithmic}
\end{algorithm}

There are many possibilities to make our algorithm more efficient, e.g., by utilising additional properties of the graph $G$ or by reducing the number of tests for subgraph containment (lines~\ref{line4} to~\ref{line6}). For simplicity, however, we do not delve into these details here.

The most important part of Algorithm~\ref{algorithm} is the function \texttt{isSubgraph$(H,F)$}, which determines whether $F$ contains a copy of $H$ as subgraph. If yes, then \texttt{isSubgraph$(H,F)$} returns the Boolean value true, otherwise it returns the Boolean value false. In general, this is known as the \textit{subgraph isomorphism problem}, which is NP-complete~\cite{cook1971complexity}. A na\"{i}ve way to implement \texttt{isSubgraph$(H,F)$} is to traverse every injective mapping $f:V(H)\to{V(F)}$ and check whether $f$ preserves the edges of $H$. Of course, this approach has an exponential running time and is not sensible for any practical use. Ullmann~\cite{ullmann1976subgraph} introduced a more sophisticated algorithm for the subgraph isomorphism problem, which decreases the running time immensely compared to the na\"{i}ve approach. Although the running time of Ullmann's algorithm could still be exponential in the worst case, it is good enough for our practice.

Another tricky part in Algorithm~\ref{algorithm} is to determine $\mathrm{Hom}'(G)$, which is also an NP-complete problem in general~\cite{hell2004graphs}. However, we avoided this problem by running Algorithm~\ref{algorithm} on certain small graphs, whose minimal homomorphic images can be simply determined.

\subsection{Chromatic Ramsey numbers of two small $\boldsymbol{4}$-chromatic graphs}
We now use Algorithm~\ref{algorithm} to determine the chromatic Ramsey numbers of two small $4$-chromatic graphs. Given a finite group $(V,+)$ and a subset $S\subseteq{V}$, the \textit{Cayley graph} $\Gamma(V,S)$ is a graph whose vertex set is $V$ and whose edge set is
\begin{equation*}
\left\{\{v,v+s\}:\,v\in{V},\,s\in{S}\right\}.
\end{equation*}
The first graph that we consider here is the Cayley graph $\Gamma=\Gamma\left(\mathbb{Z}_{8},\{1,2\}\right)$, see Figure~\ref{fig2}. For the second graph we consider the \textit{Moser spindle} $M$, see Figure~\ref{fig3}. It has been proved by Paul~\cite{paul2009thesis} that $11\leq{R_{\chi}(M)}\leq13$, thus, determining $R_{\chi}(M)$ surely yields a new value for the chromatic Ramsey numbers of $4$-chromatic graphs.
\begin{figure}[h]
\begin{minipage}[t]{0.49\textwidth}
\centering
\begin{tikzpicture}[scale=0.7, main_node/.style={circle,draw,minimum size=1em,inner sep=1pt]}]

\node[main_node] (0) at (-4.8, 3.3) {0};
\node[main_node] (7) at (-4.8, 0.6) {7};
\node[main_node] (6) at (-3, -0.4) {6};
\node[main_node] (1) at (-3, 4.6) {1};
\node[main_node] (2) at (1, 4.6) {2};
\node[main_node] (5) at (1, -0.4) {5};
\node[main_node] (4) at (2.8, 0.6) {4};
\node[main_node] (3) at (2.8, 3.3) {3};

 \path[draw, thick]
(0) edge node {} (1) 
(0) edge node {} (2) 
(1) edge node {} (2) 
(1) edge node {} (3) 
(2) edge node {} (3) 
(2) edge node {} (4) 
(3) edge node {} (4) 
(3) edge node {} (5)
(4) edge node {} (5) 
(4) edge node {} (6) 
(5) edge node {} (6) 
(5) edge node {} (7) 
(6) edge node {} (7) 
(6) edge node {} (0) 
(7) edge node {} (0) 
(7) edge node {} (1);
\end{tikzpicture}
\caption{$\Gamma\left(\mathbb{Z}_{8},\{1,2\}\right)$}
\label{fig2}
\end{minipage}
\hfill
\begin{minipage}[t]{0.49\textwidth}
\centering
\begin{tikzpicture}[scale=1.4,line width=0.87pt,main_node/.style={circle,draw,minimum size=0.5em,inner sep=1pt,fill=black}]

\def\phi{{asin(1/(1.7*sqrt(3)))}}
\def\rho{1.7}

\begin{scope}[xshift=-1.25cm]
  \begin{scope}[rotate=180]
    \begin{scope}[rotate=90]
      \draw (\phi:{\rho*sqrt(3)}) -- (-\phi:{\rho*sqrt(3)});
    \end{scope}
    \foreach \th in {-\phi, \phi}
    {
      \begin{scope}[rotate=\th]
        \draw ( 0, 0) -- ( 60:\rho);
        \draw ( 0, 0) -- (120:\rho);
        \draw ( 0, {\rho*sqrt(3)}) -- ( 60:\rho);
        \draw ( 0, {\rho*sqrt(3)}) -- (120:\rho);
        \draw ( 60:\rho) -- (120: \rho);

        \node[main_node] at ( 0, 0) {};
        \node[main_node]  at ( 60:\rho) {};
        \node[main_node] at (120:\rho) {};
      \end{scope}
    }
    \begin{scope}[rotate=90]            
      \node[main_node] at (\phi:{\rho*sqrt(3)}) {};
      \node[main_node] at (-\phi:{\rho*sqrt(3)}) {};
    \end{scope}

  \end{scope}
\end{scope}
\end{tikzpicture}
\caption{The Moser spindle $M$}
\label{fig3}
\end{minipage}
\end{figure}

It is straightforward that $\chi(\Gamma)=\chi(M)=4$. Moreover, since $\Gamma$ and $M$ are both small, it is not hard to obtain their minimal homomorphic images. Indeed, to determine $\mathrm{Hom}'(M)$, one can first list all homomorphic images of $M$ (up to vertex relabeling) on at most $7$ vertices by contracting the independent sets in $M$. As $\alpha(M)=2$, this gives a small collection of graphs. Then one can easily find the minimal homomorphic images in this collection. The procedure of determining $\mathrm{Hom}'(\Gamma)$ is analogous. Eventually, we have that
\begin{equation*}
\mathrm{Hom}'(\Gamma)=\left\{K_{4},\Gamma\right\}\quad\text{and}\quad\mathrm{Hom}'(M)=\left\{K_{4},W_{5},M\right\},
\end{equation*}
where $W_{5}$ is a graph obtained by joining a vertex to a cycle of length $5$.

Running Algorithm~\ref{algorithm} on these two graphs yields the following result.
\begin{theorem}
\label{thm5}
Let $\Gamma=\Gamma\left(\mathbb{Z}_{8},\{1,2\}\right)$ and $M$ be the Moser spindle. Then we have
\begin{equation*}
R_{\chi}(\Gamma)=14\quad\text{and}\quad{R_{\chi}(M)=11}.
\end{equation*}
\end{theorem}
Theorem~\ref{thm5} finds a new $4$-chromatic graph with chromatic Ramsey number $14$, see also~\cite{paul2012ramsey,tardif2023ramsey} for more such graphs. In addition, Theorem~\ref{thm5} reveals a new value $11$ for the chromatic Ramsey numbers of $4$-chromatic graphs and hence implies Theorem~\ref{thm3}.

\section{Sparse graphs attaining the upper bound}
\label{proof2}
In this section, we prove Theorem~\ref{thm2} by constructing $k$-chromatic graphs with arbitrarily large girth and chromatic Ramsey number $R(k)$.
\subsection{An iterative construction}
Our construction here is inspired by that of Tutte (alias Blanche Descartes)~\cite{descartes1947girth} and uses a result of Erd\H{o}s and Hajnal~\cite{erdos1966chromatic} on hypergraphs. For $r\in\mathbb{N}_{>1}$, an \textit{$r$-uniform hypergraph} is an ordered pair $\mathcal{H}=(V,\mathcal{E})$, where $V$ is a finite set and $\mathcal{E}\subseteq\binom{V}{r}$. The elements of $V$ and $\mathcal{E}$ are referred to as the vertices and the hyperedges of $\mathcal{H}$, respectively. Similar to the case of graphs, we denote $\lvert\mathcal{H}\rvert:=\lvert{V}\rvert$ and $\lVert\mathcal{H}\rVert:=\lvert\mathcal{E}\rvert$.

The \textit{chromatic number} $\chi(\mathcal{H})$ of a hypergraph $\mathcal{H}$ is the smallest integer $k$, such that one can color the vertices of $\mathcal{H}$ with $k$ colors so that every hyperedge of $\mathcal{H}$ contains two vertices of distinct colors. A \textit{cycle} of length $s\in\mathbb{N}_{>1}$ in $\mathcal{H}$ is an alternating sequence
\begin{equation*}
\left(v_{0},E_{0},v_{1},E_{1},\dots,v_{s-1},E_{s-1},v_{s}:=v_{0}\right)
\end{equation*}
of distinct vertices $v_{0},\dots,v_{s-1}\in{V}$ and distinct hyperedges $E_{0},\dots,E_{s-1}\in\mathcal{E}$, such that $v_{i-1},v_{i}\in{E_{i-1}}$ for all $i\in[s]$. If we allow repeated hyperedges, then such an alternating sequence is called a \textit{hyper-cycle}. It is a simple observation that every hyper-cycle with at least two distinct hyperedges contains a subsequence that is a cycle. The \textit{girth} $\textsl{g}(\mathcal{H})$ is the length of a shortest cycle in $\mathcal{H}$. Note that if $\mathcal{H}$ contains no cycle, then $\textsl{g}(\mathcal{H})=\infty$.

\begin{proposition}[Erd\H{o}s--Hajnal~\cite{erdos1966chromatic}]
\label{pro4}
For every $k,g\in\mathbb{N}$ and $r\in\mathbb{N}_{>1}$ there exists an $r$-uniform hypergraph $\mathcal{H}$ with $\chi(\mathcal{H})\geq{k}$ and $\textsl{g}(\mathcal{H})\geq{g}$.   
\end{proposition}

Let $g,m\in\mathbb{N}$. We iteratively construct the graph $T_{i}=T_{i}(g,m)$ for $i\in\mathbb{N}$. To begin with, we let $T_{1}$ be an empty graph on two vertices. For $i>1$, by Proposition~\ref{pro4} there exists a $\lvert{T_{i-1}}\rvert$-uniform hypergraph $\mathcal{H}_{i}$ such that $\chi(\mathcal{H}_{i})\geq{m}$ and $\textsl{g}(\mathcal{H}_{i})\geq{g}$. Now we construct $T_{i}$ as follows. First we take an independent set $U_{i}$ of $\lvert\mathcal{H}_{i}\rvert$ vertices and $\lVert\mathcal{H}_{i}\rVert$ pairwise vertex-disjoint copies of $T_{i-1}$, each of which is also vertex-disjoint from $U_{i}$. Then, we connect every copy of $T_{i-1}$ with a distinct $\lvert{T_{i-1}}\rvert$-element subset of $U_{i}$ that corresponds to a hyperedge of $\mathcal{H}_{i}$ via a perfect matching, see Figure~\ref{fig1}.

\begin{figure}[H]
\centering
\begin{tikzpicture}[x=0.7pt, y=0.7pt, yscale=-1, xscale=1, main_node/.style={circle,draw,minimum size=0.3em,inner sep=0.1pt,fill=black}]
\node at (135,73){$T_{i-1}$};
\draw   (100,80) .. controls (100,68.95) and (115.67,60) .. (135,60) .. controls (154.33,60) and (170,68.95) .. (170,80) .. controls (170,91.05) and (154.33,100) .. (135,100) .. controls (115.67,100) and (100,91.05) .. (100,80) -- cycle;
\node at (235,73){$T_{i-1}$};
\draw   (200,80) .. controls (200,68.95) and (215.67,60) .. (235,60) .. controls (254.33,60) and (270,68.95) .. (270,80) .. controls (270,91.05) and (254.33,100) .. (235,100) .. controls (215.67,100) and (200,91.05) .. (200,80) -- cycle;
\node at (335,83){$\dots$};
\node at (435,73){$T_{i-1}$};
\draw   (400,80) .. controls (400,68.95) and (415.67,60) .. (435,60) .. controls (454.33,60) and (470,68.95) .. (470,80) .. controls (470,91.05) and (454.33,100) .. (435,100) .. controls (415.67,100) and (400,91.05) .. (400,80) -- cycle;
\node at (260,160){\Large$\dots$};
\node at (300,160){\Large$U_{i}$};
\node at (340,160){\Large$\dots$};
\draw   (100,148) .. controls (100,143.58) and (103.58,140) .. (108,140) -- (466,140) .. controls (470.42,140) and (474,143.58) .. (474,148) -- (474,172) .. controls (474,176.42) and (470.42,180) .. (466,180) -- (108,180) .. controls (103.58,180) and (100,176.42) .. (100,172) -- cycle;
\draw    (110,83) -- (110,160);
\draw    (120,83) -- (140,160);
\node at (135,83){$\dots$};
\draw    (150,83) -- (170,160);
\draw    (210,83) -- (140,160);
\draw    (220,83) -- (170,160);
\node at (235,83){$\dots$};
\draw    (250,83) -- (230,160);
\draw    (410,83) -- (390,160);
\draw    (420,83) -- (410,160);
\node at (435,83){$\dots$};
\draw    (450,83) -- (450,160);
\node[main_node] at (110,83) {};
\node[main_node] at (110,160) {};
\node[main_node] at (120,83) {};
\node[main_node] at (140,160) {};
\node[main_node] at (150,83) {};
\node[main_node] at (170,160) {};
\node[main_node] at (210,83) {};
\node[main_node] at (220,83) {};
\node[main_node] at (250,83) {};
\node[main_node] at (230,160) {};
\node[main_node] at (410,83) {};
\node[main_node] at (390,160) {};
\node[main_node] at (420,83) {};
\node[main_node] at (410,160) {};
\node[main_node] at (450,83) {};
\node[main_node] at (450,160) {};
\end{tikzpicture}
\caption{The constructed graph $T_{i}$}
\label{fig1}
\end{figure}

We continue by pointing out some important properties of the graph $T_{i}$.
\begin{lemma}
\label{lem5}
Let $k,g,m\in\mathbb{N}$ with $m\geq{k}$ and $T_{i}=T_{i}(g,m)$ be defined as above with $i\in[k]$. Then we have $\chi(T_{i})=i$ and $\textsl{g}(T_{i})\geq3g$.
\end{lemma}
\begin{proof}[Proof of Lemma~\ref{lem5}]
We prove this by induction on $i\in[k]$. Since $T_{1}$ is an empty graph with two vertices, we have $\chi(T_{1})=1$ and $\textsl{g}(T_{i})=\infty$. Namely, the assertion holds for $i=1$.

For $i>1$, assume that $\chi(T_{i-1})=i-1$ and $\textsl{g}(T_{i-1})\geq3g$. Let $\mathcal{H}_{i}$ be the hypergraph which we use to construct $T_{i}$, in particular $\chi(\mathcal{H}_{i})\geq{m}$ and $\textsl{g}(\mathcal{H}_{i})\geq{g}$. Recall that $T_{i}$ consists of $\lVert\mathcal{H}_{i}\rVert$ vertex-disjoint copies of $T_{i-1}$ and an independent set $U_{i}$ of $\lvert\mathcal{H}_{i}\rvert$ vertices. As $\chi(T_{i-1})=i-1$, we can properly color the vertices of these copies of $T_{i-1}$ with $i-1$ colors and then color $U_{i}$ with the $i$th color. This gives us $\chi(T_{i})\leq{i}$. On the other hand, suppose one can properly color $V(T_{i})$ with $i-1$ colors. Since $\chi(\mathcal{H}_{i})\geq{m}\geq{k}>i-1$, there exists a $\lvert{T_{i-1}}\rvert$-element subset $U'\subseteq{U_{i}}$ which corresponds to a hyperedge of $\mathcal{H}_{i}$, such that all vertices in this subset obtain the same color. Let $T'$ be the copy of $T_{i-1}$ that is matched with $U'$ in $T_{i}$. Then the vertices of $T'$ must be properly colored with at most $i-2$ colors, which implies that $\chi(T_{i-1})\leq{i-2}$, a contradiction to our induction hypothesis.

It remains to show that $\textsl{g}(T_{i})\geq3g$. Fix an arbitrary cycle $W$ in $T_{i}$. Either the vertices of $W$ are completely contained in some copy of $T_{i-1}$, or some of its vertices are in $U_{i}$. In the former case, the length of $W$ is at least $\textsl{g}(T_{i-1})\geq3g$. In the latter case, let $X\dot\cup{Y}$ be the vertex set of $W$, where $X\subseteq{U_{i}}$ and $Y\subseteq{V(T_{i})\backslash{U_{i}}}$. It is easy to see that $\lvert{X}\rvert\geq2$, as the neighborhood of any $x\in{X}$ is an independent set in $T_{i}$. Relabel the elements of $X$ as $\left\{x_{1},\dots,x_{\lvert{X}\rvert}\right\}$ and let $x_{0}:=x_{\lvert{X}\rvert}$, so that as the cycle $W$ is traversed, the elements of $X$ appear in the order $x_1,\dots,x_{\lvert{X}\rvert}$. Since $X$ is an independent set in $T_{i}$, for every $j\in[\lvert{X}\rvert]$, $x_{j-1},x_{j}$ in $W$ must be connected via the matching edges and the edges inside some copy of $T_{i-1}$, implying that there exists a $\lvert{T_{i-1}}\rvert$-element subset of $U_{i}$ containing $x_{j-1},x_{j}$ which corresponds to a hyperedge of $\mathcal{H}_{i}$. Moreover, $x_{0},x_{1}$ and $x_{1},x_{2}$ in $W$ must be connected via two different copies of $T_{i-1}$. Therefore, $X$ corresponds to the vertex set of a hyper-cycle in $\mathcal{H}_{i}$ with at least two distinct hyperedges. Recall that every hyper-cycle with at least two distinct hyperedges contains a cycle. Namely, $\lvert{X}\rvert\geq\textsl{g}(\mathcal{H}_{i})\geq{g}$. Note that every vertex $x\in{X}$ is followed in $W$ by at least two vertices from $Y$. Hence, we have $\lvert{Y}\rvert\geq2\lvert{X}\rvert\geq2g$, completing the proof.
\end{proof}

\begin{lemma}
\label{lem6}
Let $k,g,m\in\mathbb{N}$ with $m\geq{k}$ and $T_{i}=T_{i}(g,m)$ with $i\in[k]$. Then for any $F\in\mathrm{Hom}(T_{i})$, either $F$ contains a copy of $K_{i}$ or $F$ has at least $m$ vertices.
\end{lemma}
\begin{proof}[Proof of Lemma~\ref{lem6}]
We shall apply induction on $i\in[k]$. As $T_{1}$ is a graph on two vertices, every $F\in\mathrm{Hom}(T_{1})$ contains at least one vertex, i.e., a copy of $K_{1}$. Thus, the assertion holds for $i=1$.

For $i>1$, assume that every homomorphic image of $T_{i-1}$ either contains a copy of $K_{i-1}$ or has at least $m$ vertices. Now fix an arbitrary $F\in\mathrm{Hom}(T_{i})$ and suppose $\lvert{F}\rvert<m$. We shall argue that such $F$ must contain a copy of $K_{i}$. Let $\mathcal{H}_{i}$ and $U_{i}$ be the hypergraph and the corresponding independent set which we use to construct $T_{i}$. Let $f:V(T_{i})\to{V(F)}$ be a homomorphism from $T_{i}$ to $F$. Since $f$ maps any two adjacent vertices to distinct vertices of $F$, $f$ naturally defines a proper coloring of $V(T_{i})$ with $\lvert{F}\rvert<m$ colors. Recall that $\chi(\mathcal{H}_{i})\geq{m}$, namely, there exists a $\lvert{T_{i-1}}\rvert$-element subset $U'\subseteq{U_{i}}$ which corresponds to a hyperedge of $\mathcal{H}_{i}$, such that $f$ maps $U'$ to a single vertex $v\in{V(F)}$. Let $T'$ be the copy of $T_{i-1}$ that is matched with $U'$ in $T_{i}$. Observe that $f$ maps $T'$ to a subgraph $F'\subseteq{F}$, moreover, $F'\in\mathrm{Hom}(T_{i-1})$. Then, as $\lvert{F'}\rvert\leq\lvert{F}\rvert<m$, $F'$ contains a copy of $K_{i-1}$ by our induction hypothesis. Since $T'$ is connected with $U'$ via a perfect matching, $v$ is the common neighbor of all vertices of $F'$. Namely, $F$ contains a copy of $K_{i}$.
\end{proof}
\subsection{Proof of Theorem~\ref{thm2}}
\begin{proof}[Proof of Theorem~\ref{thm2}]
Given any $k,g\in\mathbb{N}$, we want to prove that there is a graph $G$ with $\chi(G)=k$ and $\textsl{g}(G)\geq{g}$, such that $R_{\chi}(G)=R(k)$. To this end, we let $m=R(k)$ and $G:=T_{k}(g,m)$. Then by Lemma~\ref{lem5} we have $\chi(G)=k$ and $\textsl{g}(G)\geq3g$. It remains to show that $R_{\chi}(G)=R(k)$.

Since $\chi(G)=k$, it follows from~\eqref{equ1} that $R_{\chi}(G)\leq{R(k)}$. On the other hand, let $N=R(k)-1$. There exists a $2$-coloring of $E(K_{N})$ avoiding a monochromatic copy of $K_{k}$. By Lemma~\ref{lem6}, every homomorphic image of $G$ either contains a copy of $K_{k}$ or has at least $m>N$ vertices. Accordingly, under this $2$-coloring $K_{N}$ contains no monochromatic copy of any $F\in\mathrm{Hom}(G)$. Therefore, we have $R_{\chi}(G)>N=R(k)-1$.
\end{proof}

\section{Concluding remarks and open problems}
\label{conclusions}
In Theorem~\ref{thm4} we showed that there are many graphs of the same chromatic number, whose chromatic Ramsey numbers are distinct. An alternative way to phrase Theorem~\ref{thm4} is that there is an absolute constant $C>0$, such that for any function $f:\mathbb{N}\to\mathbb{R}_{\geq0}$ with $(k-1)^{2}+1\leq{f(k)}\leq{R(k)}$, we can find a sequence $(G_{k})_{k\in\mathbb{N}}$ of graphs with $\chi(G_{k})=k$ and
\begin{equation}
\label{equ8}
Cf(k)\cdot\max\left\{\frac{\log{f(k)}}{k\log{k}},\,\left(\frac{\log{f(k)}}{k}\right)^{2}\right\}<R_{\chi}(G_{k})\leq{f(k)}.
\end{equation}
This parallels a result of Pavez-Sign\'{e}, Piga, and Sanhueza-Matamala~\cite{pavez2023ramsey} regarding the classical Ramsey numbers. They showed that for any non-decreasing function $f:\mathbb{N}\to\mathbb{R}_{\geq0}$ with $k\leq{f(k)}\leq{R(k)}$, there is a sequence $(G_{k})_{k\in\mathbb{N}}$ of graphs such that $\lvert{G_{k}}\rvert=k$ and $Cf(k)\leq{R(G_{k})}\leq{f(k)}$, where $C>0$ is an absolute constant.

This ``redundant'' term $\max\left\{\frac{\log{f(k)}}{k\log{k}},\,\left(\frac{\log{f(k)}}{k}\right)^{2}\right\}$ in~\eqref{equ8} arises from Proposition~\ref{pro1}. Let $G$ be a non-empty graph and $F$ be obtained by deleting a vertex from $G$. Intuitively, one would expect that $R_{\chi}(F)$ is at least a constant fraction of $R_{\chi}(G)$. However, to the best of our knowledge there are no other results in this direction. A conjecture of Conlon, Fox, and Sudakov~\cite[Conjecture 5.1]{conlon2020short} asserted that $R(F)\geq\Omega\left(R(G)\right)$. Although this conjecture was recently refuted by Wigderson~\cite{wigderson2024ramsey}, it is not clear whether its analogue for chromatic Ramsey numbers could be true.

We hereby raise the following two questions. First, from~\eqref{equ1} we may deduce that the integers $3,4,7,8,9$ cannot be realized as chromatic Ramsey numbers. It is then natural to ask:
\begin{question}
Are there other positive integers not realizable as chromatic Ramsey numbers? Are there infinitely many of them?
\end{question}
Theorem~\ref{cor1} states that there are at least $\Omega(k)$ distinct values for chromatic Ramsey numbers of $k$-chromatic graphs. However, there is a large gap between the obtained lower bound $\Omega(k)$ and the trivial upper bound $R(k)-(k-1)^{2}$. This motivates our second question:
\begin{question}
What is the order of magnitude of $\,\left\lvert\left\{R_{\chi}(G):\,\chi(G)=k\right\}\right\rvert$?
\end{question}

The extremal function $\mathrm{ex}^{(2)}(n,G)$ that we consider, could be naturally generalized to the multicolor case with several forbidden monochromatic graphs. For $r\in\mathbb{N}$ and fixed graphs $G_{1},\dots,G_{r}$, the function $\mathrm{ex}^{(r)}(n;G_{1},\dots,G_{r})$ is defined as the largest number of edges in an $n$-vertex graph whose edges can be colored with $r$ colors, such that there is no monochromatic copy of $G_{i}$ with the $i$-th color for all $i\in[r]$. Similarly, one can define $\pi^{(r)}(G_{1},\dots,G_{r}):=\lim_{n\to\infty}\mathrm{ex}^{(r)}(n;G_{1},\dots,G_{r})/\binom{n}{2}$. As demonstrated by Hancock, Staden, and Treglown~\cite[Section 5.3]{hancock2019ramsey}, $\pi^{(r)}(G_{1},\dots,G_{r})$ is expressible in terms of $R_{\chi}(G_{1},\dots,G_{r})$, which is the smallest integer $N$ such that there is a graph $F$ with $\chi(F)=N$, so that any $r$-coloring of $E(F)$ results in a monochromatic copy of $G_{i}$ with the $i$-th color for some $i\in[r]$. Therefore, the study of $\pi^{(r)}(G_{1},\dots,G_{r})$ again focuses on understanding the multicolor chromatic Ramsey numbers.

It is worth mentioning that Keevash, Saks, Sudakov, and Verstra\"{e}te~\cite{keevash2004turan} introduced a different notion of multicolor Tur\'{a}n numbers. They define the function $\mathrm{ex}_{r}(n,G)$ to be the largest number of edges in an $n$-vertex multigraph whose edges can be colored with $r$ colors avoiding a rainbow copy of $G$, i.e., a copy of $G$ whose edges all have distinct colors. The authors of~\cite{keevash2004turan} showed that when $r$ is sufficiently large, an optimal construction would be taking $r$ copies of a fixed extremal $G$-free graph. See also~\cite{chakraborti2024rainbow,li2024turan} for various results.

Lastly, we note that Liu, Pikhurko, and Sharifzadeh~\cite{liu2019edges} considered a similarly sounding function $\mathrm{nim}(n,G)$, which is the largest number of edges in $K_{n}$ not contained in any monochromatic copy of $G$ over all $2$-colorings of $E(K_{n})$. Nevertheless, $\mathrm{nim}(n,G)$ and $\mathrm{ex}^{(2)}(n,G)$ are really two different functions. Already in the case of cliques, Keevash and Sudakov~\cite{keevash2004edges} extended an earlier result of Erd\H{o}s~\cite{erdos1997graph} and Pyber~\cite{pyber1986clique} and proved that $\mathrm{nim}(n,K_{k+1})=\mathrm{ex}(n,K_{k+1})=t(n,k)$ for sufficiently large $n$, where $t(n,k)$ denotes the number of edges in the Tur\'{a}n graph $T(n,k)$. On the other hand, Bialostocki, Caro, and Roditty~\cite{bialo1990zero} showed that $\mathrm{ex}^{(2)}(n,K_{k+1})=\left(1+o(1)\right)t\left(n,R(k+1)-1\right)$.\\

\noindent\textbf{Acknowledgements.} The research was partially supported by the DFG grant FKZ AX 93/2-1. The authors thank the anonymous referees for their careful reading and invaluable comments, which have improved the presentation of this paper.

\end{document}